\documentclass{siamart}

\usepackage{amssymb}

\usepackage{latexsym}

\newtheorem{thm}{Theorem}[section]

\newtheorem{lem}[thm]{Lemma}
\newtheorem{prop}[thm]{Proposition}
\theoremstyle{definition}

\theoremstyle{remark}
\newtheorem{rem}[thm]{Remark}

\numberwithin{equation}{section}

\newcommand{\R}{\mathbb{R}}

\usepackage{amsfonts}
\usepackage{graphicx}
\usepackage{epstopdf}
\usepackage{algorithmic}
\ifpdf
  \DeclareGraphicsExtensions{.eps,.pdf,.png,.jpg}
\else
  \DeclareGraphicsExtensions{.eps}
\fi

\newsiamremark{remark}{Remark}
\newsiamremark{hypothesis}{Hypothesis}
\crefname{hypothesis}{Hypothesis}{Hypotheses}
\newsiamthm{claim}{Claim}

\headers{IVP for RSW}{}

\title{Initial Value Problem for one-dimensional rotating shallow water equations}

 \author{
 Nabil Bedjaoui\thanks{LAMFA UMR CNRS 7352, Université de Picardie Jules Verne
 (\email{nabil.bedjaoui@u-picardie.fr}).}
 \and 
 Vivien Desveaux\thanks{LAMFA UMR CNRS 7352, Université de Picardie Jules Verne
 (\email{vivien.desveaux@u-picardie.fr}).}
 \and
 Olivier Goubet\thanks{Laboratoire Paul Painlevé CNRS UMR 8524, et équipe projet INRIA PARADYSE, Université de Lille
 (\email{olivier.goubet@univ-lille.fr}).}
  \and
  Alice Masset\thanks{LAMFA UMR CNRS 7352, Université de Picardie Jules Verne
  (\email{alice.masset@u-picardie.fr}).}
  }

\usepackage{amsopn}

\begin{document}

\maketitle

\begin{keywords}
	hyperbolic system, Coriolis force, diffusive regularization, symmetrization.
\end{keywords}

\begin{AMS}
	Primary 35L45;  Secondary  35M11, 76D03.
\end{AMS}

\begin{abstract}
In this article we address some issues related to the initial value problems
for a rotating shallow water hyperbolic system of equations and the diffusive regularization of this system. For initial data close to the solution at rest, we establish the local existence and the uniqueness of a solution to the hyperbolic system, as well as the global existence of a solution to the regularized system. In order to prove this, we use suitable variables that symmetrize the system.
\end{abstract}


\maketitle

\section{Introduction}

The present work is devoted to the study of the initial value problem for one-dimensional shallow water equations with Coriolis force,
 also called rotating shallow water (RSW) equations. The RSW equations are used in oceanography and meteorology to model geophysical motions at large scale where the Coriolis force due to the Earth rotation plays a fundamental role.  These equations are given in conservative form by
\begin{equation}\label{eq:RSW2D}
\begin{aligned}
 h_t + (hu)_x + (hv)_y & = 0,\\
 (hu)_t + \left(hu^2 + \frac{gh^2}{2}\right)_x  + (huv)_y &  = fhv , \\
 (hv)_t + (huv)_x + \left( hv^2 + \frac{gh^2}{2} \right)_y & = -fhu.
\end{aligned}
\end{equation}
\noindent  In this article we are interested in the one-directional reduction
of these equations that reads,
\begin{equation}\begin{aligned}\label{eq:RSW}
h_t + (hu)_x & =0, \\
(hu)_t + \left( hu^2 + \frac{gh^2}{2} \right)_x & = fhv, \\
(hv)_t + (huv)_x & = -fhu,
\end{aligned}\end{equation}

\noindent where $h$ denotes the fluid height, $u$ the horizontal velocity, and $v$ the transverse one.
The fluid at rest solution of these equations reads $h=\bar h$ and $u=v=0$.
The Coriolis force $f$ which depends on $x$ is such that $f, f'$ and $f''$ are in $L^\infty(\R)$.
The gravity $g$ is constant. These equations are supplemented with initial data $h_0,u_0,v_0$.

A large literature is devoted to the numerical aspects of shallow water equations. We can mention the important work \cite{Hydrostatic} that introduces well-balanced schemes based on hydrostatic reconstruction for shallow water equations with topography. For works specifically dedicated to RSW equations, we refer to \cite{FirstRSW} which describes numerical schemes for the 1D system,  or more recently to \cite{Chertock} for the 2D system. Furthermore, model derivation and physical aspects can be found in \cite{LivreZeitlin}. However, not so much has been done about the theoretical study of RSW equation. Then, we address in the present article the initial value problem.

In the following we assume that $h_0-\bar{h},u_0,v_0$ belong to $(H^2(\R))^3$ and that $h_0(x)\geq \underline h>0$ where $\underline h$ is a constant. Our main result reads as follows

\begin{thm} \label{thm:E&U RSW}
Assume $h_0-\bar{h},u_0,v_0$ are in $H^2(\R)$ and $h_0 \geq \underline{h} > 0$.
There exists a unique classical solution $(h-\bar h,u,v)^T$ in $C([0,T_0]; H^2(\R)^3 )$ of the system \eqref{eq:RSW},
where $T_0$ depends on the initial data as $T_0 \simeq \Vert (h_0-\bar h,u_0,v_0)^T \Vert_{H^2(\R)^3}^{-1}$; moreover $h$ remains positive.
\end{thm}
The proof relies on a suitable change of unknowns in the equations
and on a suitable diffusive approximation (depending on a small parameter $\varepsilon$) of the new system of equations. 

We can easily adapt this proof for RSW equations with non constant topography $z$ in $H^3(\R)$.
These equations read
\begin{equation}\begin{aligned}
h_t + (hu)_x & =0, \\
(hu)_t + \left( hu^2 + \frac{gh^2}{2} \right)_x & = fhv - g h z_x, \\
(hv)_t + (huv)_x & = -fhu.
\end{aligned}\end{equation}
We do not detail it here for the sake of conciseness.

The second main result of this article states that the solution
of the regularized diffusive approximation are global in time
if the initial data is close enough to the fluid at rest solution.

The paper is organized as follows. Section \ref{proofofthemainthm}
is devoted to the proof of the main theorem.
We first introduce the change of variable and the diffusive approximation.
We then prove the local well-posedeness for this regularized system,
and then pass to the limit when $\varepsilon$ goes to $0$.
In the following section, we prove for $\varepsilon>0$ fixed
a global existence result. The last section validates the change of variable
since the condition  $h>0$ holds true.

We complete this introduction with some notations.
If $H^m(\R)$ is the classical Sobolev space whose scalar elements
and their first $m$ derivatives are in $L^2$, then $\mathbb{H}^m(\R)=H^m(\R)^3$.
For $m=0$ we write respectively $L^2(\R)$ and $\mathbb{L}^2(\R)$. We also denote $L^\infty(\R)^3$ by $\mathbb{L}^\infty(\R)$.

\section{Proof of Theorem \ref{thm:E&U RSW}}\label{proofofthemainthm}

\subsection{Introducing new variables and diffusive approximation}

 As long as $h>0$, instead of the unknowns $(h,hu,hv)$ in \eqref{eq:RSW}, we rather use the unknowns $\lambda := 2\sqrt{g h}, u, v$.
 We set $\bar\lambda=2\sqrt{ g  \bar{h} }$.
Then system \eqref{eq:RSW} yields
\begin{equation}\label{eq:RSW SYM}
(V-E)_t+ S(V)(V-E)_x+  F \times ( V - E) = 0,
\end{equation}

\noindent where $$V=\begin{pmatrix}\lambda \\ u \\ v \end{pmatrix}, \;
F = \begin{pmatrix} f \\ 0 \\ 0 \end{pmatrix},
\; E=\begin{pmatrix} \bar{\lambda} \\ 0 \\ 0 \end{pmatrix},$$

$$ {\rm and}  \; S(V)=S=\begin{pmatrix}
u & \frac{\lambda}{2} & 0 \\
\frac{\lambda}{2} & u & 0 \\
0 & 0 & u \end{pmatrix}. $$

This change of variable turns the nonlinear part of the equations into
a symmetrizable non-linearity, that is more suitable for the computations with an hyperbolic system
(see \cite{GRT2} and the references therein).
The price to pay is to check that $h$ remains positive along the process.
This will be addressed in Section \ref{plus}.
There are other possible methods to symmetrize the system. For example we can multiply the system \eqref{eq:RSW} by the entropy derivative,
see \cite{GRT2}.

We now introduce a regularized version of \eqref{eq:RSW SYM} adding a diffusive term in the right hand side
as follows
\begin{equation}\label{eq:RSW SYM REG}
(V-E)_t+ S(V)(V-E)_x+  F \times ( V - E) = \varepsilon (V-E)_{xx}.
\end{equation}
with $\varepsilon > 0$. We now address the initial value problem for this regularized system.

\subsection{Solving the regularized equation}

In this section, we prove existence and uniqueness of a solution to \eqref{eq:RSW SYM REG}.

\begin{prop} \label{thm:E&U RSW SYM REG}
	For $\varepsilon>0$ fixed, there exists $T_0\eqsim  C ||V(0)-E||^{-2}_{\mathbb{H}^2}$ such that
	equation \eqref{eq:RSW SYM REG} has a unique solution $V$ in $C([0,T_0]; E+\mathbb{H}^2(\R))$.
\end{prop}

\begin{proof}

We first solve the equation applying a fixed point theorem
on a short interval of time $T_\varepsilon(0) \eqsim C\varepsilon \min(1, ||V(0)-E||_{\mathbb{H}^2}^{-2}).$
We then iterate this process to obtain a solution that is defined
on a maximal interval of time $[0,T_0)$ with $T_0> \frac{C}{||V(0)-E||_{\mathbb{H}^2}^2}$.
For later use, we define $$N(t)=\sqrt{ ||V(t)-E||^2_{\mathbb{L}^2}+||V_{xx}(t)||^2_{\mathbb{L}^2}},$$
\noindent that is a norm on $E+\mathbb{H}^2(\R)$.

\subsubsection{A fixed point theorem}

Consider first the linear evolution equation

\begin{equation}\label{mer1}
W_t =\varepsilon W_{xx},
\end{equation}

\noindent supplemented with initial data in $\mathbb{L}^2(\R)$.
Set $W(t)=\mathcal{W}(t)W(0)$ for this heat flow.
We recall the following standard result.

\begin{lem}\label{lem:operator}
	The semigroup $\mathcal{W}(t)$ satisfies $||\mathcal{W}(t)||_{\mathcal{L}(\mathbb{L}^2)}\leq 1$
	and $$||\mathcal{W}(t)||_{\mathcal{L}(\mathbb{L}^2,\mathbb{H}^1)}\leq \sqrt{1+\frac{1}{2e\varepsilon t}}.$$
\end{lem}

\begin{proof}  For a $L^2$ scalar function $\varphi$ we have
$$ \int_\R (1+|\xi|^2)\exp(-2\varepsilon t\xi^2)|\hat\varphi(\xi)|^2d\xi\leq \int_\R \left( 1+\frac{1}{2e\varepsilon t} \right)|\hat\varphi(\xi)|^2d\xi.$$
\noindent Then the result follows.
\end{proof}

We now seek a mild solution to \eqref{eq:RSW SYM REG}.
We perform a fixed point argument in $C([0,T]; \mathbb{H}^2(\R))$ to the operator $\mathcal{T}$
that maps $V-E$  to
\begin{equation}\label{f1}
\mathcal{W}(t)(V(0)-E)- \int_0^t \mathcal{W}(t-s)  \left( S(V)(V-E)_x + F \times (V-E) \right)ds.
\end{equation}

Using Lemma \ref{lem:operator}, we have
\begin{multline}\label{mer5}
 ||\mathcal{T}(V-E)(t)||_{\mathbb{H}^2 }\leq ||(V-E)(0)||_{\mathbb{H}^2 }+\\
\int_0^t \sqrt{1+\frac{1}{2e\varepsilon (t-s)}}\left( ||S(V)(V-E)_x||_{\mathbb{H}^1} + || F \times (V-E)  ||_{\mathbb{H}^1} \right)ds.
\end{multline}

\noindent  Since $H^1(\R)$ is a Banach algebra we have
\begin{equation*}\begin{split}
||S(V)(V-E)_x||_{\mathbb{H}^1}\leq C \left( ||V-E||^2_{\mathbb{H}^2}+||V-E||_{\mathbb{H}^2}\right).
\end{split}\end{equation*}

\noindent Besides we have $$||F \times (V-E)||_{  \mathbb{H}^1 } \leq c(|| f ||_{L^\infty }, ||f' ||_{L^\infty } ) || V - E ||_{ \mathbb{H}^1 }.$$
Therefore for $||V(0)-E||_{\mathbb{H}^2 }\leq \frac R 2$ we have that
\begin{multline}\label{mer6}
||\mathcal{T}(V-E)(t)||_{\mathbb{H}^2 }\leq \frac R 2
+ C(|| f ||_{L^\infty}, || f' ||_{L^\infty} ) \\
\times \left(T+\frac{\sqrt T}{\sqrt \varepsilon}\right)
 \sup_{[0,T]} \left( ||V-E||^2_{\mathbb{H}^2 } + || V-E ||_{\mathbb{H}^2} \right).
\end{multline}

\noindent Therefore for $T$ small enough depending on $R$, $f$, $f'$ and $\varepsilon$,
$\mathcal{T}$ maps the ball of radius $R$ in $C([0,T]; \mathbb{H}^2)$ into itself.
Since the map $(V,W)\mapsto S(V)W$ is bilinear, to prove
that $\mathcal{T}$ is a contraction if $T$ is small enough
is similar and then omitted for the sake of conciseness.
Thus we can apply a fix point theorem.

\subsubsection{{\it A priori} estimate}

We prove {\it a priori} bound for the solution of Proposition \ref{thm:E&U RSW SYM REG}.
	
\noindent Consider the scalar product of \eqref{eq:RSW SYM REG} with $V-E$. This leads to,
with $|V|$ being the $\R^3$ euclidean norm, and $V.W$ the corresponding scalar product
\begin{equation}\label{symog3}
\frac{d}{dt} \int_{\R} |V-E|^2dx +2\varepsilon \int_\R |V_x|^2 dx =- 2\int_\R (V-E).S (V-E)_x dx.
\end{equation}

\noindent Since $S$ is self-adjoint then
\begin{equation}\label{symog4}
- 2\int_\R (V-E).S (V-E)_x dx=\int_\R (V-E).S_x (V-E)dx.
\end{equation}

\noindent Let us observe that

$$ S_x= \begin{pmatrix}
u_x & \frac{\lambda_x}{2} & 0 \\
\frac{\lambda_x}{2} & u_x & 0 \\
0 & 0 & u_x
\end{pmatrix}. $$

\noindent Then, setting $\displaystyle |||S_x ||| = \sup_{x\in \R} || S_x ||_{\mathcal{L}(\R^3)}$ we have

\begin{equation}\label{sbound}
|||S_x|||\leq || V_x ||_{\mathbb{L}^\infty(\R)}.
\end{equation}

\noindent This comes from the fact that
\begin{equation}\label{cor0}
||S_x||_{\mathcal{L}(\R^3)}=\max \left(|u_x|, \Big\vert u_x+ \frac{\lambda_x}{2} \Big\vert,  \Big\vert u_x- \frac{\lambda_x}{2} \Big\vert \right).
\end{equation}

\noindent Then we have
\begin{equation}\label{cor3}
\frac{d}{dt} \int_{\R} |V-E|^2dx +2\varepsilon \int_\R |V_x|^2 dx \leq |||S_x||| \int_\R |V-E|^2 dx.
\end{equation}
Therefore dropping a positive term
\begin{equation}\label{symog5}
\frac{d}{dt} \int_{  \R} |V-E|^2dx \leq \left( \sup_{[0,T]} ||V_x||_{\mathbb{L}^\infty(\R)} \right)  \int_\R |V-E|^2 dx.
\end{equation}

Introduce now the stopping time
$$\tau=\inf \{ t>0 ; ||V(t)-E||_{ \mathbb{H}^2 } > 2 ||V(0)-E||_{\mathbb{H}^2 }=2M_0\}.$$

\noindent For $T\leq \tau$ we set $\sigma(T,V_x)=\sup_{[0,T]} ||V_x||_{\mathbb{L}^\infty(\R)}$.
Then we have, integrating \eqref{symog5}, that for $t\leq T$

\begin{equation}\label{f2}
||V(t)-E||^2_{\mathbb{L}^2}\leq ||V(0)-E||^2_{\mathbb{ L}^2 }  + \sigma(T,V_x) T \sup_{s\leq t}||V(s)-E||^2_{ \mathbb{L}^2 }.
\end{equation}

  Consider now the scalar product of \eqref{eq:RSW SYM REG} with $V_{4x}$.
Integration by parts yields
\begin{multline}\label{symog8}
\frac{d}{dt} \int_{ \R} |V_{xx}|^2dx +2\varepsilon \int_\R |V_{xxx}|^2 dx =-2\int_\R V_{4x}.S V_x dx \\
-2 \int_\R V_{xx} . F_{xx} \times (V-E) dx -4 \int_\R V_{xx} . F_x \times V_x { dx}.
\end{multline}

\noindent To begin with, we handle the last two terms in the right hand side
of \eqref{symog8}. Using the following interpolation inequality
$$||V_x||^2_{ \mathbb{L}^2 }\leq||V_{xx}||_{\mathbb{L}^2}||V-E||_{\mathbb{L}^2 },$$
\noindent we have
$$ -2 \int_\R V_{xx} . F_{xx} \times (V-E) dx\leq 2||f''||_{L^\infty}||V_{xx}||_{\mathbb{L}^2 } ||V-E||_{\mathbb{L}^2}, $$

\noindent and
$$ -4 \int_\R V_{xx} . F_x \times V_x\leq 4||f'||_{L^\infty} |||V_{xx}||^\frac32_{\mathbb{L}^2} ||V-E||^\frac12_{\mathbb{L}^2}. $$

\noindent Then by Young's inequality, these last two terms are bounded by
$$ ( ||f''||_{L^\infty}+3\sqrt 3||f'||_{L^\infty})(||V-E||_{\mathbb{L}^2}^2+||V_{xx}||_{\mathbb{L}^2}^2).$$

\noindent Besides, integrating by parts the first term in the right hand side of \eqref{symog8}
\begin{equation}\label{symog9}
-2\int_\R V_{4x}.S V_x dx=2\int_\R V_{3x}.S_x V_{x} dx +2\int_\R V_{3x}.S V_{2x} dx.
\end{equation}

\noindent On the one hand, using again that $S$ is self-adjoint we have
$$ 2\int_\R V_{3x}.S V_{2x} dx = -  \int_\R V_{2x}.S_x V_{2x} dx \leq |||S_x||| \int_\R |V_{xx}|^2 dx.$$

\noindent On the other hand
$$\int_\R V_{3x}.S_x V_{x} dx= - \int_\R V_{2x}.S_x V_{2x} dx -\int_\R V_{2x}.S_{2x} V_{x} dx. $$

\noindent A mere computation leads to
$$ \int_\R -V_{2x}.S_{2x} V_{x} dx \leq  \frac 32 ||V_x||_{L^\infty(\R)} \int_\R |V_{xx}|^2 dx.$$

\noindent Gathering these inequalities we have
\begin{multline}\label{symog10}
\frac{d}{dt} \int_\mathbb{R} |V_{xx}|^2 dx \leq \left( 6 \sup_{[0,T]} ||V_x||_{L^\infty(\R)} \right) \int_\R |V_{xx}|^2 dx + \\
( ||f''||_{L^\infty}+3\sqrt 3||f'||_{L^\infty})(||V-E||_{\mathbb{L}^2}^2+||V_{xx}||_{\mathbb{L}^2}^2).
\end{multline}

\noindent Combining this inequality with \eqref{f2}, we have that for $t \leq T$ this leads to
\begin{equation}\label{f4}
N(t)^2 \leq N(0)^2+ T \left( 6 \sigma(T,V_x) + 3\sqrt 3|| f' ||_{L^\infty} + || f'' ||_{L^\infty}  \right) \sup_{s\leq t} N(s)^2.
\end{equation}

\noindent  Therefore for $t\leq T$, assuming $2T \left( 6 \sigma(T,V_x) + 3\sqrt 3|| f' ||_{L^\infty} + || f'' ||_{L^\infty}  \right)\leq 1$
we have $||V(t)-E||_{\mathbb{H}^2 }\leq 2 ||V(0)-E||_{ \mathbb{H}^2 }.$
Therefore $\tau$ such that
$24\tau ||V_x||_{\mathbb{L}^\infty(\R)} \leq 1$ and $4\tau(3\sqrt 3|| f' ||_{L^\infty} + || f'' ||_{L^\infty})\leq 1$
provides a lower bound for the maximum time of existence. This completes the proof of Proposition \ref{thm:E&U RSW SYM REG}.

\end{proof}

\subsection{Passing to the limit $\varepsilon \rightarrow 0$}

Here we denote by $V^\varepsilon$ the solution of the regularized problem.
We have proved local existence until a time $T$ independent of $\varepsilon$ for any $\varepsilon > 0$ fixed. We first prove the sequence $(V^\varepsilon)_{\varepsilon > 0}$ is a Cauchy sequence in $\mathbb{L}^2(\R)$, and then establish the $\mathbb{H}^2$-regularity  for the limit $V$ of $V^\varepsilon$.

\begin{prop}
	The sequence $V^\varepsilon-E$ converges towards $V-E$  that is
	the unique solution in $C([0,T_0];  \mathbb{H}^2(\R) )$  of \eqref{eq:RSW SYM}.
\end{prop}

\begin{proof}
	
	We first prove
	
	\begin{lem}
		Consider $T_0$ the existence time of the solution that depends
		on the initial data but that is independent of $\varepsilon$.
		Then $V^\varepsilon-E$ converges in $C([0,T_0];  \mathbb{L}^2(\R)  )$ to a limit
		$V-E$.
	\end{lem}
	
	\begin{proof}
		
		Consider for $0<\eta<\varepsilon$ two solution $V^\varepsilon$ and $V^\eta$.
		Then the difference $W=V^\varepsilon-V^\eta$ is solution to
		\begin{equation}\label{m1}
		W_t+S(V^\varepsilon)W_x+(  S(V^\varepsilon)-S(V^\eta))V_x^\eta+F\times W=(\varepsilon-\eta)V_{xx}^\eta+\varepsilon W_{xx}.
		\end{equation}
		
		\noindent Considering the scalar product with $W$ leads to
		\begin{multline}\label{m2}
		 \frac12\frac{d}{dt}||W||^2_{  \mathbb{L}^2 }\leq (\varepsilon-\eta)||V_{xx}^\eta||_{  \mathbb{L}^2 }||W||_{  \mathbb{L}^2 }\\
		-\int_\R W. S(V^\varepsilon)W_xdx-\int_\R W. (  S(V^\varepsilon)-S(V^\eta))V_x^\eta dx.
		\end{multline}

		\noindent On the one hand since the map $V \mapsto S(V)$ is Lipschitzian
		from $\R^3$ into $\mathcal{L}(\R^3)$ we have that, using that
		$V^\eta$ is bounded uniformly with respect to $\varepsilon$ in $\mathbb{H}^2(\R)\subset \mathbb{W}^{1,\infty}(\R)$
		\begin{equation}\label{m3}\begin{split}
		\Big\vert \int_\R W. ( S(V^\varepsilon)-S(V^\eta))V_x^\eta dx \Big\vert \leq C ||W||^2_{  \mathbb{L}^2 }.
		\end{split}\end{equation}
		
		\noindent On the other hand using that the matrix $S$ is symmetric
		and that $V^\varepsilon$ is bounded uniformly with respect to $\varepsilon$ in  $\mathbb{H}^2(\R) \subset \mathbb{W}^{1,\infty}(\R)$
		\begin{equation}\label{m4}\begin{split}
		\Big|\int_\R W. S(V^\varepsilon)W_xdx \Big|=\frac12 \Big| \int_\R W.(S(V^\varepsilon))_xWdx \Big| \leq C ||W||^2_{\mathbb{ L}^2}.
		\end{split}\end{equation}
		
		Gathering these inequalities we obtain
				\begin{equation}\label{m5}
		\frac{d}{dt}||W||_{  \mathbb{L}^2 }\leq C \varepsilon+C||W||_{  \mathbb{L}^2 }.
		\end{equation}
		
		\noindent Therefore by the Gronwall lemma $||W(t)||_{ \mathbb{L}^2(\R)}\leq C(T_0,V_0)\varepsilon$
		and $V^\varepsilon$ is a Cauchy sequence in $C([0,T_0]; \mathbb{L}^2(\R) )$.
	\end{proof}

	We now use the classical lemma
	
	\begin{lem}
		Consider a sequence $w_\varepsilon$ that is bounded in $C([0,T_0]; H^2(\R) )$
		and that converges towards $w$ in $C([0,T_0]; L^2(\R))$.
		Then $w_\varepsilon$ converges towards $w$ in $C([0,T_0]; H^s(\R) )$
		for $s<2$, in $L^\infty([0,T_0]; H^2(\R))$ weakly star,
		and the limit $w$ is weakly continuous with values in $H^2(\R) $.
	\end{lem}
	
	\begin{proof}
		
		Since the sequence $w_\varepsilon$ is bounded in $L^\infty([0,T_0]; H^2(\R)) $
		then by Banach-Alaoglu-Bourbaki theorem \cite{brezis} $w_\varepsilon$
		converges in $L^\infty([0,T_0]; H^2(\R))$ weakly star to a function that
		is necessarily $w$. Then by interpolation
		$$ ||w-w_\varepsilon||_{H^s }\leq ||w-w_\varepsilon||^{2-s}_{L^2 }||w-w_\varepsilon||^{s}_{H^2 }, $$
		\noindent and we have the convergence in $H^s(\R) $. Since $w$ belongs to $C([0,T_0]; L^2(\R) ) \cap L^\infty([0,T_0]; H^2(\R) )$
		then by Strauss Lemma \cite{strauss} $w$ is weakly continuous with values in $H^2(\R) $.
	\end{proof}

	By interpolation estimate we deduce that $V^\varepsilon$ converge
	towards $V$ in space \\
	$E+C([0,T_0];\mathbb{H}^s(\R))$ for any $s<2$.
	Therefore we can pass to the limit in the equation and $V$ is solution
	to \eqref{eq:RSW SYM}. Moreover the solution to this equation
	is unique. Indeed, we consider two solutions of \eqref{eq:RSW SYM} $V^1$ and $V^2$. With similar arguments, the difference $W = V^1-V^2$ satisfies
 $$\frac{d}{dt} || W ||_{  \mathbb{L}^2} \leq C || W ||_{  \mathbb{L}^2},$$
 which leads to $$ || W ||_{\mathbb{L}^2} \leq 0,$$ i.e uniqueness of the solution in $E+C([0,T_0];\mathbb{H}^s(\R))$ for any $s<2$.
	
	It remains to prove that $V$ belongs to $E+C(0,T_0;\mathbb{H}^2(\R))$. We just have to prove
	that the function $t \mapsto ||V_{xx}(t)||^2_{\mathbb{L}^2}$ is continuous.
	Since $V_{xx}$ is weakly continuous with values in $\mathbb{L}^2(\R)$ then the final
	result comes promptly.
	Using that \eqref{symog10} is valid for $\varepsilon=0$
	we have that
	\begin{equation}\label{m10}
	\Big| ||V_{xx}(t)||^2_{\mathbb{L}^2(\R) }-||V_{xx}(t_0)||^2_{\mathbb{L}^2(\R) } \Big|\leq C \int_{t_0}^t \left( ||V(s)-E||^2_{\mathbb{H}^2(\R) }   \right)ds .
	\end{equation}
	
	\noindent This completes the proof of the Theorem. \end{proof}

\section{Global existence for small initial data}

The following statement asserts that if the initial data is close
enough to the fluid at rest solution, then the solution of the regularized system \eqref{eq:RSW SYM REG} exists forever.

\begin{thm}\label{thm:Global solution}
Fix $\varepsilon>0$. Assume $V(0) - E$ be in $\mathbb{H}^2(\R)$.
 Assume that  $||f'||_{L^\infty} <+\infty $.
There exists $\delta$ small enough (depending on $\varepsilon$ and on $||f'||_{L^\infty} $)
such that if $\Vert V(0) - E \Vert_{\mathbb{H}^1} \leq \delta$ then the solution lasts forever.
\end{thm}

\begin{rem}
We do not expect the result to be true in the case $\varepsilon=0$
since the solution of this hyperbolic system may develop shocks \cite{Majda}.
\end{rem}

The rest of the section is devoted to the proof of this theorem.
To begin with, we observe that it is enough to prove an $\mathbb{H}^1$ bound
for the solution, since the $\mathbb{H}^2$ regularity propagates along the flow of the solutions.
The first step in the proof is then a entropy-flux pair argument.

\subsection{Introducing an entropy-flux pair}

We claim that the functions
\begin{equation}\label{eq:def eta}
 \eta(V) = \frac{\lambda^2}{8}(u^2+v^2) + \frac{1}{2}\left( \frac{\lambda^2- \bar\lambda^2}{4}  \right)^2
\end{equation}
\noindent and
\begin{equation}\label{eq: def G}
G(V) =  \frac{\lambda^2u}{8}(u^2+v^2) + \frac{\lambda^2u}{4}\left( \frac{\lambda^2- \bar\lambda^2}{4} \right)
\end{equation}
\noindent defines an entropy-flux pair for the solution of the hyperbolic system.
Actually if $V$ is a smooth solution of equation \eqref{eq:RSW SYM}
then $$\partial_t \eta(V)+ \partial_x G(V)= 0. $$

For the solution of the regularized equation, the dissipation plays an important role
and we have
\begin{equation}\label{dipole}
 \partial_t \int_\R  \eta(V)dx + \int_{\mathbb{R}} \partial_x G(V) dx +\varepsilon \int_\R V_x.\eta''(V)V_xdx =0,
\end{equation}

\noindent where the hessian matrix of $\eta$ is given by
$$\eta''(V) = \begin{pmatrix}
	\frac{u^2+v^2}{4} + \frac{3 \lambda^2}{8} - \frac{\bar\lambda^2}{8} & \frac{\lambda u}{2} & \frac{\lambda v}{2} \\
	\frac{\lambda u}{2} & \frac{\lambda^2}{4} & 0 \\
	\frac{\lambda v}{2} & 0 & \frac{\lambda^2}{4}
\end{pmatrix}$$ and satisfies $\eta''(E) =  \frac{\bar{\lambda}^2}{4} I_3$, i.e. a constant times the identity matrix.

\begin{lem}\label{hf}
Assume $||V(0)-E||_{\mathbb{H}^1}\leq \delta$. Introduce
$$T_\delta=\inf\{ t>0; \max(||V(t)-E||_{\mathbb{L}^2}, ||V_x(t)||_{\mathbb{L}^2}>   \sqrt{\delta} \}.$$
For $\delta$ small enough and satisfying $\sqrt{\delta} \leq \frac{\bar\lambda}{2}$,  and $t\in [0,T_\delta]$ we have that by continuity
\begin{itemize}
\item $\eta''(V) \geq  \frac{\bar\lambda^2}{8} I_3$.
\item  $||\lambda -\bar\lambda||_{L^\infty}< \sqrt{ \delta }$ ( and then $\frac{\bar\lambda}{2} \leq \lambda \leq \frac32\bar\lambda$ ).
\end{itemize}
\end{lem}

\noindent
We now take advantage of this lemma. Assume $t\leq T_\delta$.
Then integrating \eqref{dipole} in time leads to
\begin{equation}\begin{split}\label{hop1}
\int_\R \left( (\lambda^2-\bar\lambda^2)^2+ 4 \lambda^2(u^2+v^2)\right)dx
+ 4 \bar{\lambda}^2 \varepsilon \int_0^t ||V_x(s)||^2_{\mathbb{L}^2 }ds\leq\\ \int_\R \left( (\lambda_0^2-\bar\lambda^2)^2+ 4\lambda_0^2(u_0^2+v_0^2)\right)dx.
\end{split}\end{equation}

\noindent Using Lemma \ref{hf}, we bound by above the right hand side of \eqref{hop1} by $$9 \bar{\lambda}^2||V(0)-E||^2_{\mathbb{L}^2 }\leq 9\bar{\lambda}^2\delta^2 .$$ Similarly we bound by below the left
hand side of \eqref{hop1} by   $$\bar{\lambda}^2||V-E||^2_{\mathbb{L}^2 }+ 4\bar{\lambda}^2 \varepsilon \int_0^t  ||V_x||^2_{\mathbb{L}^2 }ds.$$

\noindent We summarize this as

\begin{equation}\label{hop2}
||V(t)-E||^2_{\mathbb{L}^2 }+ 4 \varepsilon \int_0^t ||V_x(s)||^2_{\mathbb{L}^2 }ds\leq 9\delta^2.
\end{equation}
We infer from this inequality that  for $\delta$ small enough, we have $||V(t)-E||_{L^2(\R)} \leq  3\delta \leq  \sqrt{ \delta }$ forever,
at least as long as $||V_x(t)||_{L^2(\R)}$ satisfies the same inequality.

\subsection{Seeking an estimate for $V_x$}

In the first step we use that the solution remains close
to the rest solution in $L^\infty(\R)$. To prove that this is true
requires an $L^2$ estimate on $V_x$.

\begin{lem} \label{pro:global existence}
For $t\leq T_\delta$ we have that
$$\frac{d}{dt} \Vert V_x(t) \Vert_{\mathbb{L}^2}^2 \leq  \frac{3}{2(4\varepsilon)^\frac13} \Vert V_x(t) \Vert_{\mathbb{L}^2}^\frac{10}{3}
  + 6  \delta  \, \Vert f' \Vert_{\infty}\Vert V_x(t) \Vert_{\mathbb{L}^2}. $$	
\end{lem}

\begin{proof}
	Multipling \eqref{eq:RSW SYM REG} by $-V_{xx}$ leads to, after integration by parts
$$ \frac12\frac{d}{dt} ||V_x||^2_{\mathbb{L}^2 }+\varepsilon ||V_{xx}||^2_{\mathbb{L}^2 }=
\int_\R V_{xx}. F\times (V-E)+\int_\R V_{xx}.S(V)V_x.$$
\noindent On the one hand since $S(V)$ is symmetric
$$ \int_\R V_{xx}.S(V)V_x = -\frac12 \int_\R V_x.S_xV_x \leq \frac12 \int_\R ||S_x||_{\mathcal{L}(\R^3)}|V_x|^2. $$
\noindent Proceeding as above this is bounded by above by, appealing Sobolev embedding and Gagliardo-Nirenberg
identity
$$ \int_\R |V_x|^3 dx \leq ||V_x||_{\mathbb{L}^2}^\frac52 ||V_{xx}||_{\mathbb{L}^2}^\frac12\leq  \varepsilon   ||V_{xx}||^2_{\mathbb{L}^2}+\frac 34 \frac{||V_x||_{\mathbb{L}^2}^{\frac{10}{3}}}{(4\varepsilon)^\frac13}. $$
\noindent On the other hand,
\begin{equation*}\begin{split}
 \int_\R V_{xx}.F\times (V-E) = - \int_\R V_x.F_x\times (V-E).
\end{split}\end{equation*}	

\noindent Thus, appealing \eqref{hop2}

\begin{equation*}\begin{split}
\int_\R V_{xx}.F\times (V-E)\leq   ||f'||_{\infty} ||V-E||_{\mathbb{L}^2}||V_x||_{\mathbb{L}^2}\leq
 3\delta 
||f'||_{\infty} ||V_x||_{\mathbb{L}^2},
\end{split}\end{equation*}	
and the proof of lemma \ref{pro:global existence} is achieved.

\end{proof}

\subsection{Concluding the proof}

We infer from Lemma \ref{pro:global existence} that
$$ \frac{d}{dt} ||V_x(t)||^3_{\mathbb{L}^2}\leq \frac{9}{4 (4\varepsilon)^\frac13} \Vert V_x(t) \Vert_{{\mathbb{L}^2}}^\frac{13}{3}
+     9 \delta  \, \Vert f' \Vert_{\infty}\Vert V_x(t) \Vert_{\mathbb{L}^2}^2
$$
$$\leq \left( \frac{9}{4 (4\varepsilon)^\frac13}  (\sqrt \delta)^\frac{7}{3}
+    9 \delta \, \Vert f' \Vert_{\infty}\right)\Vert V_x(t) \Vert_{\mathbb{L}^2}^2.
$$
\noindent Integrating in time and appealing \eqref{hop2} we then have that for $t\leq T_\delta$
$$ ||V_x(t)||^3_{{\mathbb{L}^2}}\leq  \left( \frac{9}{4 (4\varepsilon)^\frac13}  ( \sqrt \delta)^\frac{7}{3}
+   9 \delta \, \Vert f' \Vert_{\infty}\right)\frac{9\delta^2 }{4\varepsilon} + \delta^3. $$
\noindent Choosing $\delta$ small enough depending on $\varepsilon$
such that the right hand side of this inequality is bounded by above by
$  ( \sqrt \delta)^3$
provides that the $L^2$ bound on $V_x$ is valid forever.
Therefore the solution exists for any time.


\section{Link between solutions of \eqref{eq:RSW} and \eqref{eq:RSW SYM}}\label{plus}

To complete the proof of theorem \ref{thm:E&U RSW}, we have to show we can perform the change of unknown $h\mapsto \sqrt h$ and vice-versa.

\begin{prop}
	Consider the solution to the limit problem {\eqref{eq:RSW}}
	defined on $[0,T_0]$. Then there exists a constant
	$\alpha$ that depends on the initial data such that
	$$ 0<\alpha \leq h(t,x) \leq \frac 1 \alpha. $$
\end{prop}

\begin{proof}
	
    From the single equation
	\begin{equation}\label{sat1}
	h_t+(hu)_x=0,
	\end{equation}
	\noindent multiplying by sgn$h$ we infer that
	\begin{equation}\label{sat2}
	|h|_t+(|h|u)_x=0.
	\end{equation}
	\noindent Therefore $\int_\R (|h(t,x)|-h(t,x))dx=\int_\R (|h_0(x)|-h_0(x))dx=0.$
	Introduce now
	$$ T=\inf\{ t\in (0,+\infty); \exists \; x ; h(t,x)\leq 0 \}. $$
	
	\noindent We prove below that $T\geq T_0$.
	Since $h_0>0$, then for $t\in (0,T)$ the
	function $h$ is positive. Introduce
	then $\ln h$. If $T<T_0$ then $\ln h$
	blows up in $L^\infty$.
	We infer from \eqref{sat1} that for $t<\min(T,T_0)$ then
	\begin{equation}\label{sat3}
	[\ln(h)]_t+u[\ln(h)]_x+u_x=0.
	\end{equation}
	\noindent Multiply then by $\ln(h)-\ln(h)_{xx}$ and integrate
	by parts to have
	\begin{multline}\label{sat4}
	\frac{d}{dt}||\ln(h)||^2_{H^1}=
	\int_R u_x(\ln(h)^2-(\ln(h)_x)^2)dx   \\
	- 2\int_\R u_{xx} \ln(h)_x dx
	- 2\int_\R  u_{x} \ln(h) dx.
	\end{multline}
	
	\noindent Therefore
	\begin{equation}\label{sat5}
	\frac{d}{dt}||\ln(h)||_{H^1}\leq  ||u_x||_{L^\infty}||\ln(h)||_{H^1}
	+2||u_x||_{H^1}.
	\end{equation}

	\noindent Integrating in time
	leads to a $H^1$-bound on $\ln h$, and then $\ln h$
	remains bounded in $L^\infty$.
	Then $T\geq T_0$.
\end{proof}

The last statement is about the link between regularity of $h-\bar h$ and $\sqrt{h}-\sqrt{\bar h}$. Indeed, the following proposition ensures the $H^2(\R)$ regularity is shared by both variables since we assumed that $h \geq \underline{h} > 0$.

\begin{prop} If $0 < \alpha \leq h \leq \frac{1}{\alpha}$. Then
	$h-\bar h$ belongs to $L^2(\R)$ if and only if $\sqrt{h}-\sqrt{\bar h}$ belongs to $L^2(\R)$.
	If $0 < \alpha \leq h$, then $h-\bar h$ belongs to $H^m(\R)$ if and only if $\sqrt{h}-\sqrt{\bar h}$ belongs to $H^m(\R)$
	for $m=1,2$.
\end{prop}

\begin{proof}
	We start by the case $m=0$. By writing $h-\bar h=(\sqrt{h}-\sqrt{\bar h})(\sqrt{ h}+\sqrt{\bar h})$, we have in one hand $$|\sqrt{h}-\sqrt{\bar h}| \leq \frac{|h-\bar h|}{\sqrt{\bar h}},$$ and in an other hand $$|h-\bar h| \leq |\sqrt{h}-\sqrt{\bar h}|(\sqrt{\bar h}+||\sqrt{h}||_{L^\infty(\R)}),$$ which proves the equivalence.
	
	\noindent For the case $m=1$ we write
	\begin{equation}\label{lille}
	(\sqrt{h})_{x} = \frac { h_x }{2\sqrt{h}}.
	\end{equation}
	Since $h$ or $\sqrt h$ belongs to $H^1(\R)\subset L^\infty(\R)$
	we have that $h$ is bounded in $L^\infty(\R)$. The {the} case $m=0$ applies.
	Then $\sqrt{h}$ is bounded from below and from above and
	we infer from \eqref{lille} that $h_x$ in $L^2(\R)$
	and $(\sqrt h)_x$ in $L^2(\R)$ are equivalent.
	
	\noindent If $m=2$, then we use
	\begin{equation}\label{d2}
	(\sqrt{h})_{xx} = \frac{h_{xx}}{2\sqrt{h}} - \frac{h_x^2}{4h^{3/2}}.
	\end{equation}
	
	Since $h$ is bounded from below and from above we have that
	$ (\sqrt h)_{xx}$ belongs to $L^2(\R)$ if and only if
	$h_{xx}-\frac{(h_x)^2}{2h}$ belongs to $L^2(\R)$.
	Due to the embedding $H^1(\R)\subset L^4(\R)$
	we then have that $h \in H^2(\R)$ implies that
	$h_{xx}-\frac{(h_x)^2}{2h}$ belongs to $L^2(\R)$.
	Conversely we have that
	\begin{equation}\label{lille2}
	\Big\Vert h_{xx} - \frac{(h_x)^2}{2h} \Big\Vert^2_{L^2} = || h_{xx} ||^2_{L^2} - \int_\R \frac{h_{xx} (h_x)^2}{h} dx + \int_\R \frac{(h_x)^4}{4h^2} dx.
	\end{equation}

	\noindent Integrating by parts $$-\int_\R \frac{h_{xx}(h_x)^2}{h}dx=-\frac13 \int_\R \frac{(h_x)^4}{h^2}dx.$$
	Appealing the case $m=1$ we have that $(\sqrt h)_x $ in $H^1(\R)$ then $\frac{(h_x)}{\sqrt h}$ in $L^4(\R)$.
	Therefore $$ ||h_{xx}||^2_{L^2(\R)}\leq \Big\Vert h_{xx}-\frac{(h_x)^2}{2h} \Big\Vert^2_{L^2}+\frac{1}{12} \int_\R \frac{(h_x)^4}{h^2}dx,$$
	and the proof is completed.
	
\end{proof}

 \section{Conclusion}

In this work, we have studied the initial value problem for RSW equations. To simplify the computations, we have started by symmetrizing the system. We have used a non classical change of variable instead of the standard symmetrization methods. We had to check that $h$ remains positive in return. The link between initial and symmetric systems has been established in the last section. Then we have proved two main results.

On the first hand, we have obtained local existence and uniqueness of a solution for the symmetric system. First step was to use a fix-point theorem to prove this result for the regularized system. Then we established {\it a priori} estimates using the symmetric structure of matrix $S$, before passing to the limit.

On an other hand we have also detailed a result of global existence for the regularized system with initial data close to the rest solution. As we said before, such result can not pass to the limit because shocks can appear in finite time.

Weak solutions can take into account shocks, and possibly exist for any time. This question is well understood for scalar conservation laws (see \cite{GRT1}), but less is known about non linear hyperbolic systems.

A possible development to this work is to adapt the results we present here to prove local existence and uniqueness of a solution for the 2D RSW system.

\end{document}